\documentclass[11pt,reqno,a4paper]{amsart}

\usepackage{amsfonts}

\title[The nodal line of the Robin Laplacian]{The nodal line of the second eigenfunction of the Robin Laplacian in ${\mathbb{R}}^2$ can be closed}

\author{J. B.~Kennedy}

\dedicatory{\upshape
Group of Mathematical Physics, University of Lisbon\\
Av.~Prof.~Gama Pinto 2, 1649-003 Lisboa, Portugal\\[.5em]
\texttt{jkennedy@cii.fc.ul.pt}
}

\newtheorem{theorem}{Theorem}[section]
\newtheorem{lemma}[theorem]{Lemma}
\newtheorem{proposition}[theorem]{Proposition}

\newtheorem{conjecture}[theorem]{Conjecture}

\theoremstyle{remark}
\newtheorem{remark}[theorem]{Remark}

\numberwithin{equation}{section}

\newcommand{\R}{\mathbb{R}}
\newcommand{\N}{\mathbb{N}}
\DeclareMathOperator{\dist}{dist}
\DeclareMathOperator{\supp}{supp}

\newcommand*{\subsubset}{\subset\joinrel\subset}

\begin{document}

\begin{abstract}
We construct a multiply connected domain in $\R^2$ for which the second eigenfunction of the Laplacian with Robin boundary conditions has an interior nodal line. In the process, we adapt a bound of Donnelly-Fefferman type to obtain a uniform estimate on the size of the nodal sets of a sequence of solutions to a certain class of elliptic equations in the interior of a sequence of domains, which does not depend directly on any boundary behaviour. This also gives a new proof of the nodal line property of the example in the Dirichlet case.
\end{abstract}



%\thanks{This work was supported by grant PTDC/MAT/101007/2008 of the FCT, Portugal.}

\maketitle

\section{Introduction}
\label{sec:intro}

Let $\Omega \subset \R^N$ be a bounded, Lipschitz domain, and consider the eigenvalue problem for the Dirichlet Laplacian
\begin{equation}
	\label{eq:dirichlet}
	\begin{aligned}
	-\Delta u &=  \lambda u \quad\quad &\text{in $\Omega$}\\
	u &= 0  &\text{on $\partial \Omega$}
	\end{aligned}
\end{equation}
with its eigenvalues listed in increasing order and repeated according to their multiplicities, $0<\lambda_1<\lambda_2 \leq \lambda_3 \leq \ldots$, and the eigenvalue problem for the Robin and Neumann Laplacians
\begin{equation}
	\label{eq:robin}
	\begin{aligned}
	-\Delta u &= \mu u \quad\quad &\text{in $\Omega$}\\
	\frac{\partial u}{\partial\nu}+\beta u &= 0 &\text{on $\partial\Omega$},
	\end{aligned}
\end{equation}
with eigenvalues $0 \leq \mu_1 < \mu_2 \leq \mu_3 \leq \ldots$ again repeated according to their multiplicities. Here $\nu$ is the outward-pointing unit normal to $\partial\Omega$ and $\beta\geq 0$ is a constant, with $\beta=0$ corresponding, of course, to the Neumann problem.

Denote by $\psi$ an eigenfunction associated with the second eigenvalue of any of the above problems, $\lambda_2$ or $\mu_2(\beta)$, $\beta \in [0,\infty)$. Here we are interested in the behaviour of the nodal set of $\psi$, $\mathcal N:=\overline{\{x \in \Omega: \psi(x)=0\}}$, and the corresponding nodal domains, that is, the connected components of $\Omega \setminus \mathcal N$. In the Neumann case, it is a simple argument to show that
\begin{equation}
\label{eq:converse}
	\mathcal N \cap \partial \Omega \neq \emptyset
\end{equation}
for any bounded, Lipschitz $\Omega \subset\R^N$, and it was long conjectured that the same must be true for the Dirichlet problem \eqref{eq:dirichlet} (see, e.g., \cite{payne:73,payne:67} in $\R^2$, or \cite[Chapter~IX, Problem~45]{schoen:94} in $\R^N$). It has been shown that \eqref{eq:converse} indeed holds in the Dirichlet case for various classes of domains $\Omega$, most importantly for general convex $\Omega$ in $\R^2$ \cite{alessandrini:94,melas:92}, but also in certain cases in $\R^N$, e.g., on some ``thin" domains \cite{freitas:08,jerison:95}, or with various symmetries \cite{damascelli:00,lin:87,payne:73}.

However, rather intricate, multiply connected counterexamples to \eqref{eq:converse} have been found \cite{fournais:01,hoffmann:97}. In dimension three or higher, they can actually be chosen to be contractible (but not less intricate) \cite{kennedy:10}. Counterexamples have also been found on manifolds \cite{freitas:02} and on unbounded planar domains \cite{freitas:07}, but there is still a large amount of uncharted territory between the two sides. The key open questions at present seem to be whether \eqref{eq:converse} holds for simply connected domains in $\R^2$, and general convex domains in $\R^N$.

We also note that both positive and negative results were recently extended to the problem $-\Delta \psi = \lambda_2 |\psi|^{p-2}\psi$ with Dirichlet boundary conditions if $p$ is close to $2$ \cite{grumiau:09}.

So the issue is clearly quite delicate, making it rather unobvious as to whether the converse to \eqref{eq:converse} is possible in the Robin case \eqref{eq:robin}, especially when $\beta \to 0$ and we approach the Neumann problem. Let us write this as requiring that one of the two nodal domains be compactly contained in $\Omega$,
\begin{equation}
\label{eq:property}
	\Omega^-:=\{x \in \Omega:\psi(x)<0\} \subsubset \Omega,
\end{equation}
say, where we write $U \subsubset V$ to mean $U \subset K \subset V$ for some compact $K \subset \R^N$. Our primary result is that \eqref{eq:property} can in fact hold, and for the full range of $\beta>0$, at least in $\R^2$.

\begin{theorem}
\label{th:main}
	Fix $M>0$ and $\beta>0$. There exists a bounded, connected, open set $\Omega \subset \R^2$ with Lipschitz boundary 		and of area $M$, whose second Robin eigenvalue $\mu_2 (\Omega, \beta)$ of\eqref{eq:robin} is simple, with a 			corresponding eigenfunction $\psi$ satisfies $\{x \in \Omega: \psi(x) \leq 0\} \subsubset \Omega$.
\end{theorem}

\begin{remark}
\label{rem:various}
(i) Our domain will be constructed from a modification of the sequence of domains used in \cite{hoffmann:97}. Our proof is different from those in \cite{fournais:01,hoffmann:97}, although importantly it still relies on symmetry as in \cite{hoffmann:97}. It is not clear if the method in \cite{fournais:01} can be adapted directly to the Robin case (see Remark~\ref{rem:3:basicbound}), although higher dimensional examples should certainly exist.

(ii) Our method also works for the Dirichlet Laplacian. We will set up the proof of Theorem~\ref{th:main} so it is valid for this case, and on the same family of domains as the Robin case.

(iii) We can also asymptotically identify the location of the nodal sets of our domains (and by the same reasoning those from \cite{hoffmann:97}); see Remark~\ref{rem:3:location}.

(iv) Unlike in the Dirichlet case, the eigenvalues of the Robin Laplacian do not behave in a uniform way with respect to domain rescaling, as this also affects the parameter $\beta>0$ appearing in the boundary term. For example, if we let $\beta \to 0$ and simultaneously replace $\Omega$ with a homothetic rescaling $t\Omega$ for an appropriate $t=t(\Omega,\beta)>0$, we could ensure $\mu_2(t\Omega, \beta)$ remains constant. To achieve full generality of the counterexamples, it is necessary to consider an arbitrary, fixed area in addition to $\beta>0$.
\end{remark}

To indicate one of the complications involved in the Robin argument, we observe that no particular domain will work for all $\beta>0$; naturally, the problem occurs when we draw close to the Neumann problem (see also Remark~\ref{rem:3:basicbound}).

\begin{proposition}
\label{prop:failure}
	Let $\Omega \subset \R^N$ be any bounded, Lipschitz domain. There exists $\beta_0=\beta_0(\Omega)>0$ such that 		\eqref{eq:converse} holds on $\Omega$ if $\beta \in [0,\beta_0]$.
\end{proposition}

The reasoning is essentially the same as in the well-known Neumann case, following from the inequality $\mu_2(\Omega, 0) < \lambda_1 (\Omega)$. For completeness' sake, we will give the short proof in Section~\ref{sec:prelim}. Proposition~\ref{prop:failure} and some intuition invite an obvious question, although one we will not attempt to answer here.

\begin{conjecture}
\label{question}
Suppose \eqref{eq:property} for holds the Robin problem with parameter $\beta_1>0$ on some domain $\Omega \subset \R^N$. Then \eqref{eq:property} also holds for the Dirichlet problem on $\Omega$, as well as the Robin problem for every $\beta \in (\beta_1,\infty)$.
\end{conjecture}

Of course, Robin boundary conditions introduce other complications to the arguments. Our proof of Theorem~\ref{th:main} circumvents these to an extent by using at its core the following intermediate result, which does not depend on the boundary conditions (in particular allowing our simultaneous Dirichlet proof), and depends on fewer specific properties of the domains. It is hoped this will be of some independent interest.

\begin{theorem}
	\label{th:size}
	Suppose $\Omega\subsubset \Omega'$ are connected, open sets in $\R^N$ and we have a sequence  of $C^2$			solutions $\psi_n \not\equiv 0$ of $\Delta \psi_n + V_n \psi_n= 0$ in $\Omega'$, $V_n \in L^\infty(\Omega')$, $n \in \N$. 		If there exists some $\Lambda \geq sup_n \|V_n\|_{L^\infty(\Omega')}$ and $\kappa:= \inf_n \|\psi_n\|_{L^\infty 			(\Omega)}/ \|\psi_n\|_{L^\infty(\Omega')}>0$, then there exists a constant $C>0$ depending only on $N$, $\Lambda$, 		$\Omega$, $\Omega'$ and $\kappa$ such that
	\begin{displaymath}
	\sigma (\{x \in \overline\Omega: \psi_n(x)=0\}) \leq C
	\end{displaymath}
	for all $n \in \N$.
\end{theorem}

Here $\sigma(U) $ denotes $N-1$-dimensional surface measure of a set $U \subset \R^N$. In the definition of $\kappa$, by standard theory of solutions to elliptic equations, we could replace the $L^\infty$-norms by the corresponding $L^p$-norms for any $p \in (1,\infty)$, with $C$ suitably adjusted.

We will apply this theorem to a sequence of domains $\Omega_n$, with $\Omega \subsubset \Omega' \subsubset \bigcap_n \Omega_n$, and with $V_n$ the simple second eigenvalue and $\psi_n$ a corresponding eigenfunction of $\Omega_n$, in order to control the behaviour of the nodal sets of $\psi_n$ as the $\Omega_n$ become ``bad" in some sense (more precisely, highly symmetric). So we explicitly think of Theorem~\ref{th:size} in terms of holding on a sequence of domains. Its basis is the body of literature, possibly tracing its origins to the work of Donnelly-Fefferman \cite{donnelly:90,donnelly:88} (see also, e.g., \cite{savo:01} and the references therein), that seeks to control, on a \emph{given} domain or manifold, the size of the nodal set as a function of the eigenvalue. That is, for a sequence of eigenvalues $\lambda_n$ of $-\Delta u = \lambda u$ on a fixed smooth manifold (without boundary or with, say, Dirichlet conditions), one can find bounds $c_1\sqrt\lambda_n \leq \{\psi_n(x)=0\} \leq c_2 \sqrt\lambda_n$, with $c_1$ and $c_2$ depending only on the manifold and not on $n \geq 1$ (see the introduction to \cite{donnelly:88} for a heuristic explanation of this).

Here the approach, which seems to be new, is that we replace any (direct) dependence on the domain or manifold and the boundary behaviour with the number $\kappa>0$. We think of such estimates as in Theorem~\ref{th:size} as depending only on intrinsic properties of solutions to elliptic equations, provided they are sufficiently well-behaved in some sense (in this case as embodied by $\kappa$). Our proof is a fairly easy adaptation of existing results from \cite{hardt:89,kukavica:98}, which seem better suited to our particular situation than the arguments from \cite{donnelly:90,donnelly:88} and related papers, and also allow for a greater degree of generality.

\begin{remark}
\label{rem:volume}
(i) With care, it would probably be possible to replace the Schr\"odinger operator in Theorem~\ref{th:size} with a somewhat more general uniformly elliptic one. It should also be possible to remove the boundedness assumption on the $V_n$, if we then replace the bound $C>0$ by a suitable expression of the form $Cf(\|V_n\|_{L^\infty}+\text{osc}\,V_n)$ (as in \cite{kukavica:98}).

(ii) Under the same assumptions as in Theorem~\ref{th:size}, it should also be possible to prove in a similar spirit a local asymmetry result for the nodal domains, that is, a uniform lower bound on the local volume of each nodal domain in the vicinity of a zero, following Mangoubi \cite{mangoubi:08}. Such a result would probably be a useful complement to Theorem~\ref{th:size} when working in higher dimensions. More precisely, there should be a $C>0$, depending on the quantities in Theorem~\ref{th:size} (possibly also requiring $V_n \geq 0$), such that, for an appropriate fixed $r_0>0$, $|\{ \psi_n(x)>0\} \cap B_r(x_0)\}|/|B_r(x_0)| \geq C$ for all $n \geq 1$, $r \in (0,r_0]$ and all $x_0 \in \overline\Omega$ such that $\psi(x_0)=0$, where $B_r(x_0)$ denotes the ball of radius $r$ and centre $x_0$, and $|U|$ the $N$-dimensional volume of $U \subset \R^N$. The proof of such a result would follow that given in \cite[Sections~4 and 5]{mangoubi:08}, with the key growth bound on the eigenfunctions from \cite{donnelly:88} used in Section~5 there replaced by an equivalent one for a sequence of domains, such as \eqref{eq:4:generalcontrol} below. However, as we will not need this here, we do not explore it further.
\end{remark}

This paper is organised as follows. In Section~\ref{sec:prelim} we give our notation and some background results on the eigenfunctions of the problem \eqref{eq:robin}. We also give the elementary proof of Proposition~\ref{prop:failure}. In Section~\ref{sec:domain} we will introduce our domains and give the proof of Theorem~\ref{th:main}, based on Theorem~\ref{th:size}. This includes the Dirichlet case and the location of the nodal set, as mentioned in Remark~\ref{rem:various}(ii) and (iii). Theorem~\ref{th:size} will be proved in Section~\ref{sec:vanishing}. Section~\ref{sec:necest} is devoted to a technical result needed for Theorem~\ref{th:main}, namely, a confirmation that Theorem~\ref{th:size} is applicable in this case.

\smallskip

\noindent{\textbf{Acknowledgements.}} The author would like to thank Pedro Freitas for suggesting the problem, and for many helpful discussions, as well as Daniel Daners for advice on perturbation results for the Robin problem. This work was supported by grant PTDC/MAT/101007/2008 of the FCT, Portugal.

\section{Basic properties of the eigenfunctions}
\label{sec:prelim}

Here we will fix some basic notation and collect some results on the eigenvalues and eigenfunctions of the problem \eqref{eq:robin}. If we do not state otherwise, everything that holds for \eqref{eq:robin} also holds for \eqref{eq:dirichlet}, but as the latter is generally well-known, we will tend not to include references or proofs for it.

Depending on which is more convenient, we denote a point $x \in \R^N$ either using Cartesian coordinates  $x=(x_1,\ldots,x_N)$ or polar coordinates $0 \neq x=(r,\theta)$ with $r\in (0,\infty)$ and $\theta \in \mathbb S^{N-1}= \{x\in \R^N:|x|=1\}$, the unit sphere in $\R^N$. The $N$-dimensional volume of a set $U$ will be denoted by $|U|$, and we will use $\sigma(U)$ to denote the $N-1$-dimensional (surface) measure of  $U \subset \R^N$. We will denote by $B_r(x)$ a ball of radius $r$ centred at $x \in \R^N$, and by $A_{r,s}(x)$ the open annular region $B_r(x) \setminus \overline B_s(x)$ in $\R^N$. If $x=0$, we will write $B_r$ for $B_r(0)$ and $A_{r,s}$ for $A_{r,s}(0)$.

Now let us discuss the problems \eqref{eq:dirichlet} and \eqref{eq:robin}. We will denote by $-\Delta^D_\Omega$ and $-\Delta^\beta_\Omega$ the operators on $L^2(\Omega)$ associated with \eqref{eq:dirichlet} and \eqref{eq:robin} (for a given $\beta>0$), respectively. We always take the eigenvalue problems to be interpreted in the weak sense, as in \cite{daners:00}. We will generally denote eigenvalues of \eqref{eq:robin} by $\mu$ and of \eqref{eq:dirichlet} by $\lambda$. However, we may also use $\lambda$ to mean a generic eigenvalue that could belong to either problem; in such a case we will always note it explicitly.

By standard theory (see, e.g., \cite{daners:00}, especially Section~5), for each $\beta>0$, $-\Delta^\beta_\Omega$ is self-adjoint and has a sequence of eigenvalues $0<\mu_1<\mu_2\leq \ldots \to \infty$, where each eigenvalue is repeated according to its finite multiplicity. For each $j\geq 1$, we have $\mu_j=\mu_j(\Omega, \beta)$, although in practice we will usually drop at least the second argument, as we will fix $\beta$ throughout. The associated eigenfunctions $\{\psi_j\}_{j=1}^\infty$ after a suitable normalisation form an orthonomal basis for $L^2(\Omega)$. For each $j$, $\psi_j \in H^1(\Omega) \cap C^\infty(\Omega)$ is in fact analytic in $\Omega$ (see, e.g., \cite[Section~V.4]{dautray:88}), and if $\Omega$ is Lipschitz, then $\psi_j \in C(\overline\Omega)$ in addition (combine \cite[Corollary~5.5]{daners:00} with \cite[Corollary~2.9]{warma:06}).

The first eigenvalue $\mu_1$ is simple, and the associated eigenfunction $\psi_1$ may be chosen strictly positive in $\Omega$; by orthogonality, all the other eigenfunctions change sign in $\Omega$. Courant's nodal domain theorem \cite[ Section~IV.6]{courant:53}, still valid in this case, asserts that for each $j$, $\mathcal N_j = \overline{\{x \in \Omega: \psi_j(x)=0\}}$ divides $\{x \in \Omega: \psi_j(x) \neq 0\}$ into at most $j$ connected components. When $j=2$ this means the nodal domains $\Omega^+:= \{x \in \Omega: \psi_2(x)>0\}$ and $\Omega^-:= \{x \in \Omega: \psi_2(x)<0\}$ are connected subsets of $\Omega$. We will always drop the subscript $j$ from $\psi$ and $\mathcal N$, as we will always take $j=2$.

Let us now give the proof of Proposition~\ref{prop:failure}, a straightforward consequence of the inequality $\mu_2(0)<\lambda_1$.

\begin{proof}[Proof of Proposition~\ref{prop:failure}]
First, we note that on a given domain $\Omega \subset \R^N$, $\mu_2(\beta) \to \mu_2(0)$, the second Neumann eigenvalue, as $\beta \to 0$. This is easy to see, either using the minimax formula for the eigenvalues \cite[Chapter~VI]{courant:53} or the general theory in \cite[Chapter~VII]{kato:76} (our operators being at least holomorphic of type (B)). Next, using the theorem of \cite{filonov:04}, $\mu_2(0)<\lambda_1$, and so there exists $\beta_0>0$ depending on $\Omega$ such that
\begin{equation}
\label{eq:2:failure}
\mu_2(\beta)<\lambda_1
\end{equation}
 for all $\beta\in (0,\beta_0)$. Suppose for a contradiction that for some $\beta \in (0,\beta_0)$, $\Omega^- \subsubset \Omega$, say. Then since the eigenfunction associated with $\mu_2(\beta)=\mu_2(\Omega,\beta)$ is strictly positive on $\Omega^-$, we have $\mu_2(\Omega, \beta) = \lambda_1 (\Omega^-)$. Using domain monotonicity of the Dirichlet eigenvalues and \eqref{eq:2:failure},
\begin{displaymath}
\mu_2(\Omega,\beta)=\lambda_1(\Omega^-)>\lambda_1(\Omega)>\mu_2(\Omega,\beta),
\end{displaymath}
a contradiction.
\end{proof}

\section{The domain and proof of Theorem~\ref{th:main}}
\label{sec:domain}

Here we will introduce our sequence of domains and show using results from the other sections that they eventually satisfy the conclusion of Theorem~\ref{th:main}. For this section we restrict ourselves to $\R^2$. We fix $\beta>0$ and $M>0$, the desired volume, throughout. Our sequence will actually only have volume approaching $M$ asymptotically from above, but as $M$ is arbitrary we ignore this technicality. As noted in Remark~\ref{rem:various}(ii), the proof also works in the Dirichlet case. Since many of the details are identical to the Robin case, or require only trivial modifications, we will tend only to mention the Dirichlet case explicitly when there is a significant difference in the proof or reference.

Our starting point is with the idea underlying the domains from \cite{hoffmann:97} and \cite{fournais:01}, although we will construct a Lipschitz variant. We will take a ball $B_{R_1}$ and add to it an annulus $A_{R_2,R_3}$, $0<R_1<R_2<R_3$ (in \cite{fournais:01,hoffmann:97}, $R_1=R_2$, although taking $R_1<R_2$ was explicitly noted as a possibility in \cite[Remark~1]{hoffmann:97}). We choose the $R_i$ with the following property, which is not immediately obvious as we are considering the \emph{Dirichlet} eigenvalues of $B_{R_1}$.

\begin{lemma}
\label{lemma:3:base}
There exist $0<R_1<R_2<R_3$ such that $|B_{R_1} \cup A_{R_2,R_3}|=M$ and
\begin{equation}
	\label{eq:3:base}
	\lambda_1 (B_{R_1}) < \mu_1(A_{R_2,R_3}) < \lambda_2 (B_{R_1});
\end{equation}
if desired, these numbers may be chosen so that $$\lambda_2(B_{R_1})-\mu_1(A_{R_2,R_3})= \mu_1(A_{R_2,R_3}) - \lambda_1(B_{R_1}).$$
\end{lemma}

\begin{proof}
Choose a ball centred at the origin with volume $M$. Call its radius $R'_3>0$. We claim that given $\delta \in (0,1)$ arbitrary, there is a unique $R'_2 = R'_2(\delta) \in (0,R'_3)$ such that $\lambda_1(B_{R'_2}) = \delta \mu_1(A_{R'_2,R'_3})$. In fact this follows from an elementary argument using continuity and monotonicity of the eigenvalues with respect to the radii; as $R'_2 \to 0$, $\lambda_1 (B_{R'_2}) \to \infty$, while $\mu_1 (A_{R_2, R_3})$ decreases to $\mu_1 (B_{R'_3})$; as $R'_2 \to R'_3$, $\lambda_1$ decreases to $\lambda_1 (B_{R'_3})$ while $\mu_1 \to \infty$.

We fix such a $\delta>0$ and a corresponding $R_2:=R'_2$. We let $R'_1 \leq R_2$ and consider $B_{R'_1}$ and $A_{R'_2,R'_3}$. We simultaneously shrink $R'_1$, starting at $R_1'=R_2$, and increase $R'_3$, such that $|B_{R'_1}\cup A_{R'_2,R'_3}|$ is held constant at $M$. As this will increase $\lambda_1(B_{R'_1})$ and decrease $\mu_1(A_{R'_2,R'_3})$ continuously and monotonically, we continue until we find the unique $R_1:=R'_1$ and $R_3:=R'_3$, together with $R_2$ fixed, for which the last assertion of the lemma holds.
\end{proof}

We will write $A:=A_{R_2,R_3}$ as this will now be fixed throughout. However, we replace $B_{R_1}$ with a sequence of perturbed domains $U_n$, $n \geq 1$, such that (i) $\partial U_n$ is $C^\infty$, (ii) $B_{R_1} \subset U_n \subset B_{(1+\frac{1}{n})R_1}$ for all $n$, (iii) $\mu_1(U_n) \to \lambda_1(B_{R_1})$ and $\mu_2(U_n) \to \lambda_2 (B_{R_1})$ as $n\to\infty$, and finally, (iv) $U_n$ is symmetric with respect to rotations through angles $\theta = 2k\pi/n$, $k=0,1,\ldots,n-1$. Any such sequence would suffice; we will specify the $U_n$ explicitly using \cite[Example~5.2]{dancer:97}. We take $\partial U_n$ to be the set of points $(x_1,x_2) \in \R^2$ such that
\begin{displaymath}
\begin{aligned}
	x_1 &= R_1(1+\frac{1}{n}(\cos n^2\pi t))\cos\pi t\\
	x_2 &= R_1(1+\frac{1}{n}(\cos n^2\pi t))\sin \pi t
\end{aligned}
\end{displaymath}
for some $t \in [-1,1]$. Then (i), (ii) and (iv) are immediate, with $\partial U_n$ analytic, and (iii) follows from \cite[Corollary~4.6]{dancer:97}. In particular, for $n$ sufficiently large,
\begin{equation}
\label{eq:3:basicbound}
	\mu_1(U_n)<\mu_1(A)<\mu_2(U_n)
\end{equation}
and $|U_n \cup A| \to M$ as $n\to\infty$. We will always assume $n \geq 1$ is sufficiently large that $U_n \subset B_{R_2}$ and $U_n \cap A = \emptyset$.

Fixing $\varepsilon>0$, we now follow \cite{hoffmann:97} and open up a set of $n$ passages of width of order $\varepsilon$ linking $U_n$ and $A$. These we define as sectors of the form
\begin{displaymath}
	S_k=S_k(n,\varepsilon):=\{(r,\theta)\in\R^2:0<r<R_3, \frac{2k\pi}{n}-\varepsilon<\theta
	<\frac{2k\pi}{n}+\varepsilon\},
\end{displaymath}
$k=0,1,\ldots,n-1$, and set
\begin{equation}
\label{eq:3:omegas}
	\Omega_{n,\varepsilon}:=U_n\cup A \cup \Bigl(\bigcup_{k=0}^{n-1}S_k(n,\varepsilon)\Bigr)
\end{equation}
to be our dual-indexed sequence of interest. We will prove that these domains (possibly after a subsequence) satisfy Theorem~\ref{th:main} for $n \geq 1$ sufficiently large and $\varepsilon(n) > 0$ sufficiently small.

Then as can be verified directly, for each $n\geq 1$ and $\varepsilon>0$, $\Omega_{n,\varepsilon}$ is Lipschitz, and with area tending to $M$ as $\varepsilon\to 0$, for each fixed $n$. In fact, we could ``smooth out the corners" of $\Omega_{n,\varepsilon}$ in such a way that makes the $\Omega_{n,\varepsilon}$ $C^\infty$, while preserving their symmetries and increasing their area by a factor of $\varepsilon$, but we do not go into details.

\begin{lemma}
\label{lemma:3:omegaconv}
For each fixed $n \geq1$ large enough that $U_n\cap A=\emptyset$, we have
\begin{displaymath}
	\mu_j(\Omega_{n,\varepsilon}) \to \mu_j(U_n\cup A)
\end{displaymath}
as $\varepsilon\to 0$, for $j=1,2,3$. In particular $\mu_2(\Omega_{n,\varepsilon})\to \mu_1(A)$.
\end{lemma}

\begin{proof}
This follows directly from \cite[Corollary~3.7]{dancer:97}, as it is routine to show that our domains satisfy Assumption~3.2 there. Indeed, we may write down the compact set $K$ of capacity zero there explicitly as the set of $2n$ points where the $S_k$ intersect $U_n$ and $A$,
\begin{displaymath}
K=\{(R_j,0),(R_j,\frac{2\pi}{n}),\ldots,(R_j,\frac{2\pi(n-1)}{n}):j=1,2)\},
\end{displaymath}
 given in polar coordinates.
\end{proof}

In the Dirichlet case, convergence follows from \cite[Theorem~7.5]{daners:03}, although of course in this case we do not need to replace $B_{R_1}$ with the $U_n$.

\begin{remark}
\label{rem:3:basicbound}
The fact that we can establish \eqref{eq:3:basicbound}, that is, that there exists such a convergent sequence of domains $U_n$ whose Robin eigenvalues converge to their Dirichlet counterparts, is crucial for obtaining Theorem~\ref{th:main} independently of $\beta>0$ small (cf.~Proposition~\ref{prop:failure} and the comments around it). The reason we need \eqref{eq:3:basicbound} is that, letting $R_0>0$ denoted the radius of the ball $B_{R_0}$ with $\lambda_1(B_{R_0}) = \mu_1(A) \simeq \mu_2(\Omega_{n,\varepsilon})$, we have $B_{R_0} \subset B_{R_1} \subset \Omega_{n,\varepsilon}$. This means an interior nodal domain will not give an eigenvalue that is ``too big", allowing us to overcome the principle inherent in Proposition~\ref{prop:failure}. Of course, how large we have to take $n \geq 1$ will depend on $\beta$ and $M$. For this reason the domains in \cite{hoffmann:97} will not work directly, and of course, any higher dimensional examples based on the same principle would need a similar modification.
\end{remark}

An immediate consequence of Lemma~\ref{lemma:3:omegaconv} is the simplicity of the second eigenvalue (that is, its eigenspace has dimension one).

\begin{lemma}
\label{lemma:3:simplicity}
Given $n \geq 1$ sufficiently large, there exists $\varepsilon_0(n)>0$ such that $\mu_2(\Omega_{n,\varepsilon})$ is simple for all $\varepsilon \in (0, \varepsilon_0)$.
\end{lemma}

\begin{proof}
Combine Lemma~\ref{lemma:3:omegaconv}, \eqref{eq:3:basicbound} and the fact that the first two of these eigenvalues are simple.
\end{proof}

It is also immediate from the construction that $\Omega_{n,\varepsilon}$ is symmetric with respect to any rotation of angle $\frac{2k\pi}{n}$, $k=0,\ldots,n-1$, or equivalently, reflection in $n$ appropriately corresponding axes of symmetry through the origin. The simplicity can be used to show that the eigenfunction corresponding to $\mu_2$, which we will write as $\psi_{n,\varepsilon}$, must inherit the symmetries of $\Omega_{n,\varepsilon}$ for $\varepsilon(n)$ small enough. To do this we need the following essentially trivial, but powerful, generic results.

\begin{lemma}
\label{lemma:3:genericsym}
Suppose $\lambda$ is a simple eigenvalue, with eigenfunction $\psi$, of $\Omega \subset \R^N$, subject either to Robin or Dirichlet boundary conditions, and suppose that $\Omega$ has a reflection symmetry with respect to some hyperplane $H$. Then either $\psi$ is symmetric with respect to $H$, or $\psi(x) \equiv 0$ for all $x \in \Omega \cap H$.
\end{lemma}

\begin{proof}
Assume without loss of generality that $H = \{x_N=0\}$. Since $\psi(x_1,,\ldots,x_{N-1},-x_N)$ is also an eigenfunction of $\Omega, \lambda$, we may define a new eigenfunction by $\varphi(x):=\psi(x_1,\ldots,x_N)-\psi(x_1,\ldots,x_{N-1},-x_N)$,
so that $\varphi(x)=0$ on $\{x_N=0\}$. Now use simplicity.
\end{proof}

\begin{lemma}
\label{lemma:3:genericsym2}
Suppose $\Omega\subset\R^N$ is symmetric with respect to two non-orthogonal hyperplanes $H_1$ and $H_2$. If the second eigenvalue $\lambda_2$ (Robin or Dirichlet) is simple, with associated eigenfunction $\psi$, then $\psi$ is also symmetric with respect to $H_1$ and $H_2$.
\end{lemma}

\begin{proof}
By Lemma~\ref{lemma:3:genericsym}, it suffices to prove $\psi \not\equiv 0$ on either $H_1$ or $H_2$. Suppose $\psi \equiv 0$ on $H_1$, say. Then either $\psi$ is symmetric in $H_2$, implying $\psi\equiv 0$ also on the distinct hyperplane obtained as the reflection of $H_1$ in $H_2$, or  $\psi\equiv 0$ on $H_2$. Either way, this contradicts Courant's theorem.
\end{proof}

Combining Lemma~\ref{lemma:3:simplicity} with Lemma~\ref{lemma:3:genericsym2}, it follows immediately that $\psi_{n,\varepsilon}$ has all the symmetries of $\Omega_{n,\varepsilon}$ provided $n \geq 1$ is sufficiently large and $\varepsilon(n)>0$ is sufficiently small.

Let us now consider the nodal set $\mathcal N_{n,\varepsilon}:= \overline{\{x \in \Omega_{n,\varepsilon}: \psi_{n,\varepsilon}(x)=0\}}$. Denote by $R_0>0$ the radius of the ball $B_{R_0}$ such that $\lambda_1(B_{R_0}) = \mu_1(A)$ and by $R_{n,\varepsilon} > R_0$ the number such that $\lambda_1(B_{R_{n,\varepsilon}}) = \mu_2(\Omega_{n,\varepsilon})$.By \eqref{eq:3:base}, $R_0 < R_1$. Since Lemma~\ref{lemma:3:omegaconv} implies $R_{n,\varepsilon} \to R_0$ as $\varepsilon \to 0$, for all $n$ there exists $\varepsilon(n)>0$ such that $R_{n,\varepsilon} < R_1$ for all $\varepsilon \in (0,\varepsilon(n))$, that is, $B_{R_{n,\varepsilon}} \subset B_{R_1}$. (This is where we use the principle outlined in Remark~\ref{rem:3:basicbound}.) Now neither nodal domain can contain a ball of radius $B_{R_{n,\varepsilon}}$, as this would force $\mu_2$ to be too small. This in turn implies $\mathcal N_{n,\varepsilon}$ must intersect $B_{R_{n,\varepsilon}}$, as we show in the next lemma. For the meantime, we fix $n$ and $\varepsilon$ and just write $\mathcal N$ for $\mathcal N_{n,\varepsilon}$.

\begin{lemma}
\label{lemma:3:nodalset}
Fix $n \geq 1$ large enough and $\varepsilon>0$ and let $\mathcal N$ be as above.
\begin{itemize}
\item[(i)] The set $I:= I_{n,\varepsilon}:= \{r \in [0,R_3]: \mathcal N \cap \{x\in\R^2: |x|=r\} \neq \emptyset\} \subset \R$ of radial levels attained by $\mathcal N$ is connected and closed.
\item[(ii)] $\mathcal N \cap \{x \in \R^2: |x|=R_{n,\varepsilon}\} \neq \emptyset$, that is, $R_{n,\varepsilon} \in I_{n, \varepsilon}$.
\item[(iii)] The set $\mathcal N$ possesses the reflection symmetries of $\psi_{n,\varepsilon}$. In particular, if $\mu_2(\Omega_{n, \varepsilon})$ is simple, then $\mathcal N$ consists of $n$ appropriately rotated copies of the set $\mathcal N \cap \{x=(r,\theta) \in \R^2: \theta \in [-\frac{\pi}{n}, \frac{\pi}{n})\}$.
\end{itemize}
\end{lemma}

\begin{proof}
(i) Connectedness follows immediately from Courant's theorem, since $\psi_{n,\varepsilon}$ must have at least one more nodal domain than $I$ has connected components. That $I$ is closed is immediate since $\mathcal N$ is also.

(ii) First, as noted earlier, we cannot have $B_{R_{n,\varepsilon}}$ strictly contained in one nodal domain, say $\Omega_{n,\varepsilon}^-$. If we did, in the Dirichlet case we would have, using strict domain monotonicity of the eigenvalues, $\lambda_2(\Omega_{n,\varepsilon})=\lambda_1 (B_{R_{n,\varepsilon}}) > \lambda_1 (\Omega_{n,\varepsilon}^-) = \lambda_2(\Omega_{n,\varepsilon})$. The Robin case is more complicated because $\psi_{n,\varepsilon}$ is the first eigenfunction of a mixed Dirichlet-Robin problem on the possibly non-smooth domain $\Omega_{n,\varepsilon}^-$. That is, we may characterise $\mu_2(\Omega_{n,\varepsilon})$ as
\begin{equation}
	\label{eq:3:mixed}
	\mu_2(\Omega_{n,\varepsilon}) = \inf_{v \in H} \frac{\int_{\Omega_{n,\varepsilon}^-}|\nabla v|^2\,dx
	+\int_{\partial_e \Omega_{n,\varepsilon}^-}\beta v^2\,d\sigma}{\int_{\Omega_{n,\varepsilon}^-}v^2\,dx},
\end{equation}
with the infimum attained by $\psi_{n,\varepsilon}\in H$. Here $H=\{u \in H^1(\Omega_{n,\varepsilon}^-\cap C(\overline{\Omega_{n,\varepsilon}^-}): u=0 \text{ on } \partial_i \Omega_{n,\varepsilon}^-\}$, where we have written $\partial_e \Omega_{n,\varepsilon}^-:= \partial \Omega_{n,\varepsilon}^- \cap \partial \Omega_{n,\varepsilon}$ for the exterior boundary of $\Omega_{n,\varepsilon}^-$ and $\partial_i \Omega_{n,\varepsilon}^- := \partial \Omega_{n,\varepsilon}^- \cap \Omega_{n,\varepsilon}$ for its interior boundary.

Supposing $B_{R_{n,\varepsilon}} \subset \Omega_{n,\varepsilon}^-$, we would have $\mu_2 (\Omega_{n,\varepsilon})=\lambda_1 (B_{R_{n,\varepsilon}}) > \lambda_1(B_\eta)$, where $B_\eta:= B_{R_{n,\varepsilon}+\eta}$ , where $\eta(n,\varepsilon)>0$ is small enough so that $B_\eta \subset \Omega_{n,\varepsilon}^-$ still. Using the eigenfunction associated with $\lambda_1(B_\eta)$, call it $\varphi \in H$, as a test function in \eqref{eq:3:mixed}, this yields $\lambda_1(B_\eta) \geq \mu_2 (\Omega_{n,\varepsilon})$, a contradiction.

Conversely, we cannot have one nodal domain strictly contained in $B_{R_{n,\varepsilon}}$, since that would mean $\mu_2(\Omega_{n,\varepsilon}) = \lambda_1(\Omega_{n,\varepsilon}^-) > \lambda_1(B_{R_{n,\varepsilon}}) = \mu_2(\Omega_{n,\varepsilon})$. Since $B_{R_{n,\varepsilon}}$ neither strictly contains nor is contained in either nodal domain, $\mathcal N \cap B_{R_{n,\varepsilon}}$ and $\mathcal N \cap (\R^N \setminus B_{R_{n,\varepsilon}})$ are both non-empty. Connectedness of $\mathcal N$ from (i) now implies $\mathcal N \cap \{x \in \R^2: |x|=R_{n,\varepsilon}\} \neq \emptyset$.

(iii) is obvious.
\end{proof}

It is now easy to place a lower bound on $\sigma(\mathcal N)$ in terms of $n$. Let
\begin{displaymath}
	\rho_{n,\varepsilon}:=\sup I_{n,\varepsilon} \in [0,R_3]
\end{displaymath}
be the highest radial level achieved by $\mathcal N$. By Lemma~\ref{lemma:3:nodalset}(i), we have $[R_{n,\varepsilon}, \rho_{n, \varepsilon}] \subset I_{n, \varepsilon}$.

\begin{lemma}
\label{lemma:3:measure}
Fix $n \geq 1$ and $\varepsilon>0$ sufficiently large and small, respectively. Then
\begin{displaymath}
\sigma(\mathcal N \cap \{(r,\theta)\in\R^2:\theta\in[-\frac{\pi}{2},\frac{\pi}{2})\}) \geq \sigma(I_{n,\varepsilon})
	\geq R_{n,\varepsilon} - \rho_{n, \varepsilon}.
\end{displaymath}
\end{lemma}

\begin{proof}
Without loss of generality, we may assume $\mathcal N \cap \{x \in \R^2: |x|=r\}$ is a single point for each $r \in I_{n,\varepsilon}$, as this can only reduce the corresponding surface measure. Then $\mathcal N$ may be represented as the graph of some function $f: I_{n,\varepsilon} \to [-\pi/2, \pi/2)$ in the $(r,\theta)$-plane. It is immediate that the surface measure $\sigma(\mathcal N)$ of the graph of $f$ is greater than the measure $\sigma(I_{n,\varepsilon})$ of the domain of $f$.
\end{proof}

This yields the following key bound which will give the proof of Theorem~\ref{th:main} directly when combined with Theorem~\ref{th:size}. To that end, we consider what happens on a ball slightly smaller than $B_{R_1}$. Fix $\delta>0$ small. Then for any $n$ large enough and $\varepsilon(n)$ small enough, Lemma~\ref{lemma:3:measure} implies
\begin{equation}
\sigma(\mathcal N_{n,\varepsilon} \cap B_{R_1-\delta}) \geq n (\min\{R_1-\delta,\rho_{n,\varepsilon} \} - R_{n,\varepsilon}).
\end{equation}

We now reduce the double sequence $\Omega_{n,\varepsilon}$ to one sequence, by making for each $n$ an appropriate choice of $\varepsilon(n)>0$, small enough so that all the desired properties hold. That is, for $\Omega_n:=\Omega_{n,\varepsilon(n)}$ we have $\mu_1(\Omega_n) \to \lambda_1(B_{R_1})$, $\mu_2(\Omega_n) \to \lambda_1(A)$, the associated eigenfunction $\psi_n := \psi_{n,\varepsilon(n)}$ is simple (and possesses all the symmetries of $\Omega_n$), $R_n:= R_{n,\varepsilon(n)} \to R_0$ as $n \to \infty$, and finally (a requirement from Section~\ref{sec:necest}; see Lemma~\ref{lemma:5:onenode}), $|\Omega_n^- \cap A| \to 0$ as $n \to \infty$. We will also make the abbreviations $\mathcal N_n:=\mathcal N_{n,\varepsilon(n)}$ for the nodal set and $\rho_n:= \rho_{n,\varepsilon(n)}$ for its greatest radial level.

The following theorem allows us to apply Theorem~\ref{th:size} to our domains $\Omega_n$; despite appearing rather obvious its proof seems somewhat subtle, and we defer it until Section~\ref{sec:necest}.

\begin{theorem}
\label{th:3:necest}
Given the $\psi_n$ as above, and given $0<\delta_0<\delta_1$ sufficiently small (depending only on $B_{R_1}$ and possibly $\beta>0$ and $M>0$, and not on $n$), there exists $\kappa>0$ and a subsequence $\psi_{n_k}$ of the $\psi_n$ such that
\begin{displaymath}
	\|\psi_{n_k}\|_{L^\infty(B_{R_1-\delta_1})} \geq \kappa \|\psi_{n_k}\|_{L^\infty(B_{R_1-\delta_0})}
\end{displaymath}
for all $k \in \N$.
\end{theorem}

So applying Theorem~\ref{th:size} to this subsequence which we still denote by $\psi_n$ on $B_{R_1-\delta_1} \subsubset B_{R_1-\delta_0}$, there exists a constant $C>0$ not depending on $n$ such that
\begin{equation}
\label{eq:3:location}
n(\min\{R_1-\delta_1,\rho_n\}-R_n) \leq \sigma(\mathcal N_n \cap B_{R_1-\delta_1}) \leq C.
\end{equation}
As $n \to \infty$, since $R_n \to R_0$, this forces $\rho_n \to R_0$ also. Recalling the definition of $\rho_n$, this implies that for $n$ large enough (and possibly flipping the sign of $\psi_n$), the set $\{x \in \Omega_n: \psi_n(x) \leq 0\} \subsubset \Omega_n$. This completes the proof of Theorem~\ref{th:main}.

\begin{remark}
\label{rem:3:location}
We can say more. That is, since $R_n, \rho_n \to R_0$, it follows immediately that for all $\eta>0$ there exists $n_0 \geq 1$ such that $\mathcal N_n \subset A_{R_0-\eta,R_0+\eta}$ for all $n \geq n_0$. Put differently, as $n \to \infty$, the nodal domain $\Omega^-_n$ approaches a ball $B_{R_0}$ whose first Dirichlet eigenvalue $\lambda_1(B_{R_0}) = \mu_1(A)$ (see Lemma~\ref{lemma:3:base}, the comments after it, and Remark~\ref{rem:3:basicbound}), and in a strong sense: given $\eta>0$, $B_{R_0-\eta} \subset \Omega^-_n \subset B_{R_0+\eta}$ for all sufficiently large $n$. Clearly, the same argument and conclusion will work for the domains from \cite{hoffmann:97}.
\end{remark}

\section{The estimate of the nodal sets}
\label{sec:vanishing}

Here we will prove Theorem~\ref{th:size}, thus developing the machinery used in Section~\ref{sec:domain} to control the behaviour of the nodal line.  Our proof of Theorem~\ref{th:size} consists of two parts. In the first we use the techniques and results of I.~Kukavica \cite{kukavica:98} to obtain an upper bound on the order of vanishing of $\psi_n$ uniformly in $x \in \Omega$ and $n \in \N$ (Theorem~\ref{th:4:rhobound}). We apply this to the main result of R.~Hardt and L.~Simon \cite{hardt:89} to obtain a uniform local bound on the size of the nodal set of $\psi_n$. Using compactness of $\overline \Omega$ completes the proof. So we take the $\Omega$, $\Omega'$, $V_n$, $\psi_n$, $\Lambda$ and $\kappa$ as in the statement of Theorem~\ref{th:size}, and without loss of generality scale the $\psi_n$ so that $\|\psi_n\|_{L^\infty(\Omega')}=1$ and $\|\psi_n\|_{L^\infty(\Omega)} \geq \kappa$.

We fix some $r_0 \in (0,\dist(\partial\Omega,\partial\Omega'))$, where we also choose it less than the number $R_0>0$ from \cite[Lemma~2.1]{kukavica:98} (this depends only on $N$, $\Omega$ and $\Omega'$). In particular, this means for every $x \in \overline \Omega$ that the equation $\Delta \psi_n + V_n\psi_n=0$ holds in an open neighbourhood of $B_{r_0}(x)$, and $\|\psi_n\|_{L^\infty(B_{r_0}(x))} \leq 1$. We will follow Sections~4 and 5 of \cite{kukavica:98} closely, but with some subtle differences taking into account our current needs.

\begin{lemma}
\label{lemma:4:3spheres}
Let $0<r_1<r_2<r_3<r_0$. There exist $K,\theta>0$ depending only on $N,r_0,r_1,r_2,r_3$ such that, for any $\varepsilon>0$ and $x \in \overline\Omega$, the inequalities $\|\psi_n\|_{L^\infty(B_{r_3}(x))} \leq 1$ and $\|\psi_n\|_{L^\infty (B_{r_1}(x))} \leq \varepsilon$ imply $\|\psi_n\|_{L^\infty (B_{r_2}(x))} \leq Ke^{C\sqrt\Lambda}\varepsilon^\theta$, where $C>0$ depends only on $N$ and $r_0$.
\end{lemma}

\begin{proof}
This is \cite[Corollary~4.2]{kukavica:98} with $\|V_-\|_{L^\infty(B_R)} \leq \sqrt\Lambda$ and $\tilde v = 1$.
\end{proof}

\begin{lemma}
\label{lemma:4:chain}
Let $r_1 \in (0,r_0)$. There exist constants $A,C>0$ depending only on $\Omega,\Omega',r_0,r_1$ and $\kappa$ such that
\begin{displaymath}
	\|\psi_n\|_{L^\infty(B_{r_1}(x))} \geq \|\psi_n\|_{L^\infty(\Omega')}Ae^{-C\sqrt\Lambda}
	=Ae^{-C\sqrt\Lambda}
\end{displaymath}
under our normalisation, for all $x \in \overline\Omega$ and $n \in \N$.
\end{lemma}

\begin{proof}
Fix $x_0 \in \overline\Omega$ and set $\|\psi_n\|_{L^\infty(B_{r_1}(x_0))}=\varepsilon$. Let $y= y(n) \in \overline\Omega$ be such that $\|\psi_n\|_{L^\infty(B_{r_1}(y))} \geq \kappa$. In a standard argument, we construct a chain of balls from $x_0$ to $y$. That is, choosing $0<r_1<r_2<r_3<r_0$ as in Lemma~\ref{lemma:4:3spheres}, we choose points $x_1,x_2,\ldots,x_k = y$ such that
\begin{itemize}
\item[(i)] $x_i \in \overline\Omega$, $i=0,\ldots,k$;
\item[(ii)] $B_{r_1}(x_{i+1}) \subset B_{r_2}(x_i)$, $i=0,\ldots,k-1$.
\end{itemize}
Using Lemma~\ref{lemma:4:3spheres}, noting $\|\psi_n\|_{L^\infty(B_{r_1}(x_0))}=\varepsilon$ and $\|\psi_n\|_{L^\infty (B_{r_3} (x_0))} \leq 1$, if we let $K,\theta,C>0$ be as in the conclusion of the lemma, then $\|\psi_n\|_{L^\infty(B_{r_2})} \leq Ke^{C \sqrt\Lambda} \varepsilon^\theta$. By (ii), $\|\psi_n\|_{L^\infty (B_{r_1}(x_1))} \leq Ke^{C \sqrt\Lambda} \varepsilon^\theta$ also, and so we may apply Lemma~\ref{lemma:4:3spheres} to $x_1$ in turn. Continuing inductively,
\begin{equation}
\label{eq:4:chain}
\kappa\leq\|\psi_n\|_{L^\infty(B_{r_1}(x_k))} \leq (Ke^{C\sqrt\Lambda})^{1+\theta+\ldots+\theta^{k-1}}\varepsilon^{\theta^k}.
\end{equation}
Since $\overline\Omega$ is bounded, there will exist an $m=m(N,\Omega,\Omega',r_1,r_2,r_3)>0$ such that we can bound the necessary number $k$ of balls from above by $m$ uniformly in $x_0\in\overline\Omega$, that is, we may take $k \leq m$ for all $x_0$. Replacing $k$ by $m$ in \eqref{eq:4:chain} and recalling the definition of $\varepsilon$, we obtain for suitable constants $\widetilde K, \widetilde C>0$ depending on $N,\Omega,\Omega',r_1,r_2,r_3$ but not $n,x_0,\Lambda$, such that
\begin{displaymath}
\|\psi_n\|_{L^\infty(B_{r_1}(x_0))} \geq \widetilde K e^{-\widetilde C \sqrt\Lambda}
\end{displaymath}
for all $x_0\in\overline\Omega$ and all $n\in\N$.
\end{proof}

We next have a slight variant of the order of vanishing result of \cite[Corollary~4.3]{kukavica:98}.

\begin{lemma}
\label{lemma:4:ov}
Let $0<\rho_1<\rho_2<r_0$ and $x_0\in\overline\Omega$ and assume there exist constants $\gamma>0$ and $K_1>0$ such that
\begin{displaymath}
	\|\psi_n\|_{L^\infty(B_{\rho_2}(x_0))}\leq K_1\Bigl(\frac{\rho_2}{\rho_1}\Bigr)^\gamma\|\psi_n\|_{L^\infty(B_{\rho_1}
	(x_0))}
\end{displaymath}
for all $n\in\N$. Then there exist $K_2=K_2(\rho_1,\rho_2,\gamma,\Lambda,N,r_0)>0$ and $K_3=K_3(\rho_1,\rho_2,N,r_0)>0$ such that
\begin{displaymath}
	\|\psi_n\|_{L^\infty(B_r(x_0))}\geq K_2 r^{K_3(\gamma+\sqrt\Lambda+1)}
\end{displaymath}
for all $r\in (0,\rho_1)$ and $n\in\N$.
\end{lemma}

\begin{proof}
We may follow the proof of \cite[Corollary~4.3]{kukavica:98} exactly, noting that the final line of the proof is the conclusion we want (and slightly stronger than the actual statement of the corollary).
\end{proof}

\begin{theorem}
\label{th:4:rhobound}
There exist $r_1>0$ small and $C_1,C_2>0$ depending only on $N,\Omega,\Omega'$ (and $r_0$), with $C_1$ also depending on $\Lambda$, such that
\begin{equation}
	\label{eq:4:rhobound}
	\|\psi_n\|_{L^\infty(B_\rho(x))}\geq C_1 \rho^{C_2(1+\sqrt\Lambda)}
\end{equation}
for all $x\in\overline\Omega$, $\rho\in (0,r_1)$ and $n\in\N$.
\end{theorem}

\begin{proof}
Again, we follow the proof of \cite[Theorem~5.1]{kukavica:98}. We first note that
\begin{equation}
\label{eq:4:etabound}
\frac{\|\psi_n\|_{L^\infty(B_{\rho_1}(x))}}{\|\psi_n\|_{L^\infty(B_{\rho_2}(x))}} \geq \eta \text{ implies }
\|\psi_n\|_{L^\infty(B_\rho(x))} \geq C_1\rho^{C_2(1+\sqrt\Lambda-\ln\eta)}
\end{equation}
for $\rho<\rho_1<\rho_2<r_0$ if $\eta \in (0,1)$, where $C_2>0$ does not depend on $n,x,\Lambda,\rho$ and $C_1>0$ does not depend on $n,x,\rho$.

To see this, fix $\eta\in (0,1)$, choose $0<\rho_1<\rho_2<r_0$ and set
\begin{displaymath}
	\gamma = -(\ln\eta)/\ln\frac{\rho_2}{\rho_1} > 0;
\end{displaymath}
we may write $\gamma = -a\ln\eta$, $a=a(\rho_1,\rho_2)$. Rearranging, and recalling the assumption on $\eta$ in \eqref{eq:4:etabound},
\begin{displaymath}
	\|\psi_n\|_{L^\infty(B_{\rho_2(x_0)})} \leq \Bigl(\frac{\rho_2}{\rho_1}\Bigr)^\gamma
	\|\psi_n\|_{L^\infty(B_{\rho_1(x_0)})}.
\end{displaymath}
Applying Lemma~\ref{lemma:4:ov} yields \eqref{eq:4:etabound}, with $C_1$ and $C_2$ having the correct dependences (writing $\gamma=-a\ln\eta$, we absorb the $a$ into the constant $C_2$).

So we are done if we can show that \eqref{eq:4:etabound} holds for some $\eta>0$ and $0<\rho_1<\rho_2<\rho$, independent of $n$ and $x$. But this follows immediately from Lemma~\ref{lemma:4:chain}: there exist $A,C>0$ such that, for any appropriate $\rho_2>\rho_1$ and any $x\in\overline\Omega$ and $n\in\N$,
\begin{equation}
\label{eq:4:generalcontrol}
\frac{\|\psi_n\|_{L^\infty(B_{\rho_1}(x))}}{\|\psi_n\|_{L^\infty(B_{\rho_2}(x))}} \geq
\frac{\|\psi_n\|_{L^\infty(B_{\rho_1}(x))}}{\|\psi_n\|_{L^\infty(\Omega')}} \geq Ae^{-C\sqrt\Lambda}.
\end{equation}
Combining this with \eqref{eq:4:etabound}, together with a suitable rearrangement and change of constants yields \eqref{eq:4:rhobound}.
\end{proof}

We now combine Theorem~\ref{th:4:rhobound} with the results of \cite{hardt:89}. To that end, set $d:= \lfloor C_2(1+\sqrt \Lambda) \rfloor + 1$. We introduce the weighted $L^2$-norm used in \cite{hardt:89}:
\begin{displaymath}
	\|u\|_r = \|u\|_{r,x}:= r^{-\frac{N}{2}}\|u\|_{L^2(B_r(x))}.
\end{displaymath}
This is equivalent to the $L^\infty$-norm under certain conditions. In one direction we always have the trivial bound
\begin{equation}
\label{eq:4:trivial}
	\|u\|_{r,x}\leq C(N)\|u\|_{L^\infty(B_r(x))}
\end{equation}
for a dimensional constant $C>0$. The bound in the other direction, for solutions to elliptic equations, comes from \cite[Theorem~8.17 or 9.20]{gilbarg:83}. That is, any $H^1$-solution $u$ of $\Delta u+V_n u=0$ in $\Omega$ satisfies, for $B_{2R}(y) \subset \Omega$,
\begin{equation}
\label{eq:4:nontrivial}
	\|u\|_{L^\infty(B_R(y))}\leq CR^{-\frac{N}{2}}\|u\|_{L^2(B_{2R}(y))},
\end{equation}
where $C=C(N,\sqrt\Lambda R)$, $\Lambda \geq \|V_n\|_{L^\infty(\Omega)}$. In our case, we may absorb $R^{-N/2}$ into the constant $C$, by replacing $R$ by $\textrm{diam}(\Omega')$, say. Applying this to $\psi_n$, combining it with Theorem~\ref{th:4:rhobound}, and making an appropriate rescaling and concomitant adjustment of constants (including writing $d$ in place of $C_2(1+\sqrt\Lambda)$), there exists some $r^*>0$ independent of $n,x,\Lambda$ such that
\begin{equation}
\label{eq:4:hsov}
	r^{-d}\|\psi_n\|_{r,x} \geq A,
\end{equation}
for all $r \in (r,r^*)$, where $A>0$ does not depend on $n\in \N$, $x\in\overline\Omega$ or $r$.

\begin{lemma}
\label{lemma:4:boundedrm}
Suppose $r_1 \in (0,r^*)$ fixed satisfies \eqref{eq:4:hsov}. Then there exists $m=m(r_1)$ such that, for each $n\in\N$ and $x \in \overline \Omega$, there exists $r=r(n,x) \in [\frac{r_1}{2^m},r_1]$ such that
\begin{equation}
\label{eq:4:boundedrm}
	\|\psi_n\|_{r,x}<2^{d+1}\|\psi_n\|_{\frac{r}{2},x}.
\end{equation}
\end{lemma}

\begin{proof}
Let $C=C(N)$ be the constant from \eqref{eq:4:trivial}. We will choose $m$ to be
\begin{displaymath}
	m:=\lfloor \log_2(CA^{-1}{r_1}^{-d})\rfloor+1,
\end{displaymath}
with $A$ and $d$ as in \eqref{eq:4:hsov}. Without loss of generality, we may assume $A\leq C^{-1}$ so that $m\geq 1$. If \eqref{eq:4:boundedrm} fails for all such $r$, choosing $r=r_1,\frac{r_1}{2},\ldots,\frac{r_1}{2^m}$, by iteration we have
\begin{displaymath}
\begin{split}
\|\psi_n\|_{r,x} &\geq 2^{d+1}\|\psi_n\|_{\frac{r_1}{2},x}\geq\ldots\geq (2^{d+1})^m\|\psi_n\|_{\frac{r_1}{2^m},x}\\
	&\geq (2^{d+1})^m\Bigl(\frac{r_1}{2^m}\Bigr)^d A = \frac{{r_1}^d}{2^m}A
\end{split}
\end{displaymath}
using \eqref{eq:4:hsov}. Noting that $\|\psi_n\|_{r,x}\leq C$ by \eqref{eq:4:trivial} and our normalisation, this implies $C \geq A ({r_1}^d)/(2^m)$, that is,
\begin{displaymath}
m \leq \log_2(CA^{-1}{r_1}^{-d}),
\end{displaymath}
a contradiction.
\end{proof}

We can now prove our main local result. Theorem~\ref{th:size} follows from this immediately, by taking a suitable open covering of the compact set $\overline \Omega$.

\begin{theorem}
\label{th:4:localsize}
There exist constants $d>0$ and $\rho_0>0$ depending only on $\Omega$, $\Omega'$, $N$, $\Lambda$ and $\kappa$ such that
\begin{displaymath}
\sigma(B_\rho(x) \cap \{y\in\Omega':\psi_n(y)=0\}) \leq c(N)d\rho^{N-1}
\end{displaymath}
for all $\rho \in (0,\rho_0)$, all $x\in\overline\Omega$ and all $n\in\N$.
\end{theorem}

\begin{proof}
We choose $r_1>0$ sufficiently small so that \eqref{eq:4:hsov} holds for $r_1$, and that $$\sqrt\Lambda {r_1}^2 \leq \Bigl(\frac{\varepsilon_0}{d^{2N+3}}\Bigr)^{3d},$$ where $\varepsilon_0=\varepsilon_0(N) \in (0,1/2]$ is a dimensional constant from \cite[Theorem~1.7]{hardt:89}; note that $r_1$ is independent of $n$ and $x$. Note also that $B_{r_1}(x) \subset \Omega'$ for all $n\in\N$ and $x\in\overline\Omega$. We obtain $m(r_1)$ as in Lemma~\ref{lemma:4:boundedrm}, for this $r_1$. Now fix $n\in\N$ and $x\in\overline\Omega$ arbitrary.

It follows that for this $n$ and $x$, there exists $r=r(n,x) \in [\frac{r_1}{2^m},r_1]$ for which \eqref{eq:4:boundedrm} holds. Hence we may apply \cite[Theorem~1.7]{hardt:89} to $\psi_n$ on $B_r(x)$, using \eqref{eq:4:boundedrm} for $R$ in (1.5) there. Setting
\begin{displaymath}
\rho_{n,x}:=\Bigl(\frac{\varepsilon_0}{d^{2N+3}}\Bigr)^{3d}r \geq \rho_0:=\Bigl(\frac{\varepsilon_0}{d^{2N+3}}\Bigr)^{3d}
\frac{r_1}{2^m} >0,
\end{displaymath}
we have
\begin{displaymath}
	\sigma(B_\rho(x)\cap \{y\in\Omega':\psi_n(y)=0\}) \leq c(N)d\rho^{N-1}
\end{displaymath}
for all $\rho \leq \rho_{n,x}$. In particular, this holds for all $\rho\leq\rho_0$, which is independent of $x \in \overline\Omega$ and $n\in\N$.
\end{proof}

\section{Proof of Theorem~\ref{th:3:necest}}
\label{sec:necest}

Here we take the same notation and assumptions as in Section~\ref{sec:domain}. Although the conclusion of Theorem~\ref{th:3:necest} is intuitively obvious, its proof does not seem to be straightforward. The problem seems to arise because, under any normalisation of the eigenfunctions $\psi_n$ on $\Omega_n$, we expect them to converge to zero on $B_{R_1}$: the trick is to show that they do so in such a fashion that $\kappa>0$. So to prove the theorem, we first study the nodal domain concentrated in $B_{R_1}$, say $\Omega_n^- = \{x\in \Omega_n: \psi_n<0\}$, and the function $\psi_n^-$, given by $-\psi_n$ on $\Omega_n^-$ and extended by $0$ on the remainder of $\Omega_n$. The proof will be more complicated in the Robin case than in the Dirichlet case, and where there is a significant difference we split the proof accordingly. We first impose another condition on the choice of $\varepsilon(n)$ from Section~\ref{sec:domain}.

\begin{lemma}
\label{lemma:5:onenode}
After a suitable normalisation of the eigenfunctions, for each fixed $n$ we have $|\Omega_{n,\varepsilon}^- \cap A| \to 0$ as $\varepsilon \to 0$.
\end{lemma}

\begin{proof}
Denote by $\psi_A>0$ the first eigenfunction of $A$. If we extend $\psi_A$ by $0$, then $\psi_A \in C^\infty(\overline{U_n\cup A})$ is the second eigenfunction of $U_n\cup A$, identically $0$ in $U_n$ and strictly positive in $A$. By Lemma~\ref{lemma:5:conv} below, the eigenfunctions $\psi_{n,\varepsilon}$ of $\Omega_{n,\varepsilon}$ extended by $0$ satisfy $\psi_{n,\varepsilon} \to \psi_A$ in $L^2(\R^2)$ as $\varepsilon\to 0$.

Assume for a contradiction that there exists $\delta >0$ such that $|\Omega_{n,\varepsilon}^- \cap A| \geq \delta$ for all $\varepsilon>0$ sufficiently small. Choose $R_2<\rho_2<\rho_3<R_3$ such that $|A \setminus A_{\rho_2, \rho_3}| < \delta/2$. Then there exists $K=K(\rho_2,\rho_3)>0$ such that $\psi_A(x) \geq K$ on $A_{\rho_2, \rho_3}$; moreover, for each $\varepsilon>0$, we must have $|\Omega_{n,\varepsilon}^- \cap A_{\rho_2, \rho_3}| \geq \delta/2$. This implies that for every $\varepsilon>0$,
\begin{displaymath}
\begin{split}
	\|\psi_{n,\varepsilon}-\psi_A\|_{L^2(\R^2)}^2 &\geq \int_{\Omega_{n,\varepsilon}^- \cap A_{\rho_2,\rho_3}}
	|\psi_A - \psi_{n,\varepsilon}|^2\,dx\\
	&\geq K^2|\Omega_{n,\varepsilon}^- \cap A_{\rho_2, \rho_3}| \geq K^2\frac{\delta}{2} \not\to 0,
\end{split}
\end{displaymath}
contradicting $\psi_{n,\varepsilon}\to\psi_A$ in $L^2(\R^2)$.
\end{proof}

Thus one of the nodal domains must have essentially trivial intersection with the outer annulus. When choosing the $\varepsilon(n)$ to reduce to one index $\Omega_n=\Omega_{n,\varepsilon(n)}$, we specify in addition to the other conditions that $|\Omega_n^- \cap A| \to 0$ as $n\to\infty$. Let us now prove the convergence result on the eigenfunctions used in the previous lemma. While this is a specific example of a general domain perturbation result, we will only state it for this special case.

\begin{lemma}
\label{lemma:5:conv}
Let $\psi_{n,\varepsilon}$, $\psi_A$ be the second eigenfunctions of $\Omega_{n,\varepsilon}$ and $U_n \cup A$, respectively, normalised appropriately and extended by zero to functions in $L^2(\R^2)$. Then for each fixed $n \geq 1$, $\psi_{n,\varepsilon} \to \psi_A$ as $\varepsilon \to 0$.
\end{lemma}

\begin{proof}
In the Robin case, we apply \cite[Corollary~3.7]{dancer:97} to the $\Omega_{n,\varepsilon}$, valid here as was argued in the proof of Lemma~\ref{lemma:3:omegaconv}, this time to obtain convergence of the eigenprojections in $L^2(\R^2)$. The conclusion follows since the corresponding eigenvalues are simple for $\varepsilon>0$ sufficiently small (see Lemma~\ref{lemma:3:simplicity}).
The Dirichlet case is even easier since $\Omega_{n,\varepsilon} \to U_n \cup A$ in the sense of Mosco (see, e.g., \cite{daners:03}; combine Proposition~7.4 with Corollary~4.2 there).
\end{proof}

From now on, we will revert to considering the single-indexed sequence $\Omega_n$ introduced in Section~\ref{sec:domain}, rather than the $\Omega_{n,\varepsilon}$.

\begin{lemma}
\label{lemma:5:wherenodeis}
There exist $C_1>0$ and $n_1\geq 1$ such that
\begin{displaymath}
	|\Omega_n^-\cap B_{R_1}| \geq C_1
\end{displaymath}
for all $n \geq n_1$.
\end{lemma}

\begin{proof}
Here we distinguish between the Dirichlet and Robin cases. In the Dirichlet case, the Faber-Krahn inequality implies that $\lambda_n \geq \lambda_1 (B_n)$, where $B_n$ is a ball such that $|B_n|=|\Omega_n^-|$. Noting that $\lambda_n \to \lambda_1(A)$, this means that, given $\delta>0$ arbitrary, there exists $n_1\geq 1$ such that $\lambda_1(B_n)\leq \lambda_1(A)+\delta$ for all $n \geq n_1$. Since $\lambda_1(B_n)$ increases monotonically to $\infty$ as $|B_n|\to 0$, this implies a uniform lower bound on $|B_n| = |\Omega_n^-|$.

In the Robin case, we use the same argument, except that we cannot apply the Robin Faber-Krahn inequality directly to obtain $\mu_n \geq \mu_1(B_n)$. Instead, we use exactly the same trick as in \cite[Section~3]{kennedy:09}, where it was proved that for any $\Omega$ Lipschitz, if $|B|=|\Omega^-|$, then $\mu_2(\Omega) \geq \mu_1(B)$. We can now repeat the Dirichlet proof verbatim.
\end{proof}

The next lemma contains the core of the proof, and despite appearing quite obvious seems to require the most work, since, as mentioned, under the normalisation $\|\psi_n\|_{L^2(\R^2)}=1$, we expect $\psi_n \to 0$ on the set $B_{R_1}$ where $\psi_n^-$ is concentrated. As usual, we assume $\psi_n^-$ is extended by $0$ on $\R^2 \setminus \overline{\Omega_n^-}$.

\begin{lemma}
\label{lemma:5:core}
Normalise $\psi_n^-$ so that $\|\psi_n^-\|_{L^2(\R^2)}=\|\psi_n^-\|_{L^2(\Omega_n^-)}=1$. Then, possibly after passing to a subsequence, $\|\psi_n^-\|_{L^2(B_{R_1})} \to 1$ as $n\to\infty$. Moreover, for each $r \in (0,R_1)$, there exist $\delta = \delta(r) \geq 0$ and $n_0 = n_0(r) \geq 1$ such that
\begin{displaymath}
	\|\psi_n^-\|_{L^2(B_r(0))} \geq 1-\delta
\end{displaymath}
for all $n \geq n_0$, with $\delta(r)\to 1$ as $r\to R_1$.
\end{lemma}

\begin{proof}
We first prove the easier Dirichlet case, which illustrates the underlying ideas more clearly. So we start this case by noting that $\psi_n^- \in H^1_0 (\Omega_n^-) \cap H^1(\R^2)$ is the first Dirichlet eigenfunction of $\Omega_n^-$. Considering $\psi_n^- \in H^1_0 (B_{R_3})$ (recalling $\Omega_n\subset B_{R_3}$) and observing that
\begin{displaymath}
	\|\nabla \psi_n^-\|_{L^2(B_{R_3})}^2=\lambda_n \|\psi_n^-\|_{L^2(B_{R_3})}^2 = \lambda_n
\end{displaymath}
(since this is true on $\Omega_n^-$ and the integrands are zero on $B_{R_3} \setminus \Omega_n^-$), we have that $\{\psi_n^-\}$ is bounded in $H^1_0(B_{R_3})$ and hence there exists $u \in H^1(\R^2)$ such that $\psi_n^- \rightharpoonup u$ weakly in $H^1(\R^2)$ and strongly in $L^2(\R^2)$; in particular, $\|u\|_{L^2(\R^2)}=1$.

The key point is that the support of $u$, $\supp u \subset B_{R_1}$. To see this, we note that $\Omega_n^- \cup B_{R_1} \to B_{R_1}$ in the sense of Mosco (see, e.g., \cite[Theorem~7.5]{daners:03}), since $\Omega_n^- \cup B_{R_1} \supset B_{R_1}$, $B_{R_1}$ has smooth boundary and $|\Omega_n^- \setminus B_{R_1}| \to 0$ by Lemma~\ref{lemma:5:wherenodeis}. By definition of Mosco convergence, $u \in H^1_0(B_{R_1})$. In particular, $\|u\|_{L^2(\R^2)}=\|u\|_{L^2(B_{R_1})}=1$.

For each $r\in (0,R_1)$ let $\delta=\delta(r,u)\geq 0$ be the number such that $\|u\|_{L^2(B_r)}=1-\delta/2$. By the monotone convergence theorem, $\delta\to 0$ as $r\to R_1$. Now choose $n_0 = n_0(\delta(r)) = n_0(r) \geq 1$ such that $\|\psi_n^- - u\|_{L^2(\R^2)} \leq \delta/2$ for all $n \geq n_0$. Using the reversed triangle inequality, for all $n\geq n_0$ we have
\begin{displaymath}
	\|\psi_n^-\|_{L^2(B_r)} \geq \|u\|_{L^2(B_r)}-\|\psi_n^- -u\|_{L^2(\R^2)} \geq 1-\delta.
\end{displaymath}

In the Robin case, we again extend $\psi_n^-$ by $0$ to obtain a function in $L^2(\R^2)$ (but not $H^1(\R^2)$), which in a slight abuse of notation we still denote by $\psi_n^-$. However, standard results as in \cite[Chapter~7]{gilbarg:83} imply $\psi_n^- \in H^1(\Omega_n)$ with $\|\psi_n^-\|_{L^2(\Omega_n)}=1$; moreover, using $\psi_n^-$ as a test function in the weak form of $-\Delta\psi_n = \mu_n\psi_n$ in $\Omega_n$ gives
\begin{equation}
\label{eq:5:robinnormbound}
	\mu_n\|\psi_n^-\|_{L^2(\Omega_n)}^2=\|\nabla\psi_n^-\|_{L^2(\Omega_n)}^2+\beta\|\psi_n^-\|_{L^2
	(\partial\Omega_n)}^2.
\end{equation}
In particular, since the left hand side is uniformly bounded in $n$, the $\psi_n^-$ have bounded norm in $H^1(\Omega_n)$ and trace with uniformly bounded norm in $L^2(\partial\Omega_n)$.

We first show that $\psi_n^-\to 0$ in the $L^2$-norm outside $B_{R_1}$. Here we make use again of the trick from \cite{kennedy:09}, combined with a variant of the Mosco convergence argument from the Dirichlet case. That is, we first construct an appropriate smooth set $\widetilde \Omega_n$ containing $\Omega_n^- \setminus B_{R_1}$, with $|\widetilde \Omega_n| \to 0$ as $n \to \infty$; we will then show $\|\varphi \psi_n^-\|_{L^2(\widetilde\Omega_n)} \to 0$ for earch fixed $\varphi \in C_c^\infty (\R^N \setminus \overline B_{R_1})$. So for each $n$, we choose $R_2 < \rho_{2,n} < \rho_{3,n} < R_3$ such that
\begin{displaymath}
	|A_{R_2,\rho_{2,n}} \cup A_{\rho_{3,n},R_3}| \leq \frac{1}{3n}
\end{displaymath}
and, also choosing $r_n \in (0,R_1)$ such that $|A_{r_n,R_1}| \leq 1/(3n)$, we set
\begin{displaymath}
	\widehat\Omega_n:=(\Omega_n \cap A_{r_n,\rho_{2,n}})\cup A_{\rho_{3,n},R_3}\cup (\Omega_n^-\cap A).
\end{displaymath}
That is, we take the passages linking $U_n$ with $A$, thin ``strips" near the boundary of $U_n$ and $A$, and the part of $\Omega_n^-$ inside $A$. This set has a boundary which may be written as the disjoint union of three open and closed parts:
\begin{displaymath}
	\partial\widehat\Omega_n=\partial\Omega_n \cup \partial B_{r_n} \cup (\partial\widehat\Omega_n \cap A),
\end{displaymath}
where the first two are smooth (Lipschitz or better), but the latter may not be. So we approximate $\widehat\Omega_n$ by a Lipschitz domain $\widetilde\Omega_n \supset \widehat\Omega_n$, such that $\partial\widetilde\Omega_n = \partial\Omega_n \cup \partial B_{r_n} \cup \Gamma_n$, with $\Gamma_n \subsubset A$ of class $C^\infty$, and such that $|\widetilde \Omega_n \setminus \widehat \Omega_n | \leq 1/(3n)$.

In particular, this means that (i) $\Omega_n^-\setminus B_{R_1} \subset \widetilde\Omega_n \subset \Omega_n$, (ii) $\partial \Omega_n \subset \partial \widetilde \Omega_n$, (iii) $\widetilde\Omega_n$ is Lipschitz, and (iv) $|\widetilde \Omega_n| \leq 1/n + |\Omega_n \cap A_{R_1,R_2}| \to 0$ as $n \to \infty$. (Here we also use that $|U_n\setminus B_{R_1}| \to 0$ and that the total area of the passages $S_k \cap B_{A_{R_1,R_2}}$ (cf.\eqref{eq:3:omegas}) also goes to $0$.)

Now fix $\varphi \in C_c^\infty(\R^2\setminus \overline B_{R_1})$ arbitrary. We will show $\|\varphi \psi_n^-\|_{L^2(\widetilde \Omega_n)} \to 0$ as $n\to\infty$. Observing that $\psi_n^- \in H^1(\widetilde\Omega_n)$ by (i), we also have $\varphi\psi_n^- \in H^1(\widetilde\Omega_n)$ with $\supp (\varphi\psi_n^-) \subsubset \R^2\setminus \overline B_{R_1}$. We consider the first Robin eigenvalue of $\widetilde\Omega_n$. By the variational characterisation, for each $n$,
\begin{equation}
\label{eq:5:zerotoinfty}
\mu_1(\widetilde\Omega_n) \leq \frac{\int_{\widetilde\Omega_n}|\nabla\psi_n^-|^2\,dx+\int_{\partial\widetilde\Omega_n}
	\beta|\varphi\psi_n^-|^2\,d\sigma}{\int_{\widetilde\Omega_n}|\varphi\psi_n^-|^2\,dx}.
\end{equation}
Since $|\widetilde\Omega_n|\to 0$, we may apply the Robin Faber-Krahn inequality (see, e.g., \cite[Theorem~1.1]{bucur:10}) to the Lipschitz domain $\widetilde\Omega_n$ to see that $\mu_1(\widetilde\Omega_n)\to 0$ as $n \to \infty$. This forces $\|\varphi \psi_n^-\|_{L^2 (\widetilde \Omega_n)} \to 0$, provided that the numerator on the right hand side is bounded in $n$. To show this, we observe that since $\psi_n^- = 0$ on $\Gamma_n$ and $\varphi=0$ on $\partial B_{R_1}$, the boundary integral satisfies
\begin{displaymath}
\int_{\partial\widetilde\Omega_n}\beta|\varphi\psi_n^-|^2\,d\sigma=\int_{\partial\Omega_n}\beta|\varphi\psi_n^-|^2\,d\sigma
	\leq \beta\|\varphi\|_{L^\infty(\R^2)}\|\psi_n^-\|_{L^2(\partial\Omega_n)}^2,
\end{displaymath}
which is bounded in $n$ by \eqref{eq:5:robinnormbound}. For the volume integral on the right hand side of \eqref{eq:5:zerotoinfty}, we estimate
\begin{displaymath}
\begin{split}
\int_{\widetilde\Omega_n} |\nabla(\varphi\psi_n^-)|^2\,dx 
\leq \int_{\widetilde\Omega_n}|\varphi\nabla\psi_n^-|^2+|\psi_n^-\nabla\varphi|^2+2|\varphi\nabla\psi_n^-||\psi_n^-
	\nabla\varphi|\,dx\\
\leq c_0 (\|\nabla\psi_n^-\|_{L^2(\Omega_n)}^2+\|\psi_n^-\|_{L^2(\Omega_n)}^2+2\|\nabla\psi_n^-\|_{L^2(\Omega_n)}
	\|\psi_n^-\|_{L^2(\Omega_n)}),
\end{split}
\end{displaymath}
where $c_0$ is an upper bound of $\|\varphi\|_{L^\infty(\R^2)}^2$ and $\|\nabla\varphi\|_{L^\infty(\R^2)}^2$, and using (i). As $\varphi$ is fixed and $\|\psi_n^-\|_{H^1(\Omega_n)}$ is uniformly bounded by our normalisation and \eqref{eq:5:robinnormbound}, this expression is bounded in $n$. This forces $\|\varphi\psi_n^-\|_{L^2(\widetilde\Omega_n)} = \|\varphi\psi_n^-\|_{L^2(\Omega_n)} = \|\varphi\psi_n^-\|_{L^2(\R^2\setminus \overline B_{R_1})} \to 0$ (where we recall $\psi_n^-=0$ outside $\Omega_n^-$). Since $\varphi \in C_c^\infty(\R^2\setminus \overline B_{R_1})$ was arbitrary and $C_c^\infty(\R^2\setminus \overline B_{R_1})$ is dense in $L^2(\R^2\setminus \overline B_{R_1})$, this means that $\|\psi_n^-\|_{L^2(\R^2 \setminus \overline B_{R_1})} \to 0$.

This in turn implies $\|\psi_n^-\|_{L^2(B_{R_1})} \to 1$. Now, noting that $\psi_n^- \in H^1(B_{R_1})$ is bounded in the $H^1$-norm (using \eqref{eq:5:robinnormbound} and $B_{R_1} \subset \Omega_n$), we extract a weakly convergent subsequence, still denoted by $\psi_n^- \rightharpoonup u \in H^1(B_{R_1})$. Since $\partial B_{R_1}$ is smooth, Rellich's theorem can be applied and so $\psi_n^- \to u$ strongly in $L^2(B_{R_1})$. This also means $\|u\|_{L^2(B_{R_1})}=1$. We may now repeat the proof of the Dirichlet case verbatim to obtain the conclusion of the lemma.
\end{proof}

From now there is no difference between the Robin and Dirichlet cases. We next convert the $L^2$-bounds into $L^\infty$-bounds. First note that we always have the easy lower bound
\begin{displaymath}
	1-\delta(r)\leq\|\psi_n^-\|_{L^2(B_r)}\leq |B_r|^\frac{1}{2}\|\psi_n^-\|_{L^\infty(B_r)},
\end{displaymath}
with $\delta(r)$ as in Lemma~\ref{lemma:5:core}. Thus for $r\in [R_0,R_1]$, say, for all $n (=n(r))$ sufficiently large,
\begin{displaymath}
	\|\psi_n^-\|_{L^\infty(B_r)}\geq \frac{1-\delta(r)}{|B_{R_0}|^\frac{1}{2}}.
\end{displaymath}
The uniform upper bound follows from standard elliptic theory. Noting that $\psi_n^-$ extended by $0$ is a $H^1$-subsolution of $(\Delta+ \mu_n) u=0$ in $B_{R_1}$ for each $n$, \cite[Theorem~8.17]{gilbarg:83} applied to $L=\Delta+\mu_n$ and $u=\psi_n^-$, together with an elementary rescaling argument, imply that if $\Lambda>0$ is any uniform upper bound for the $\mu_n$, and if $0<r_1<r_2<R_1$, then there exists $C=C(r_1,r_2,\Lambda)>0$ such that
\begin{displaymath}
	\|\psi_n^-\|_{L^\infty(B_{r_1})} \leq C\|\psi_n^-\|_{L^2(B_{r_2})} \leq C.
\end{displaymath}
So for any $r_1<r_2$ sufficiently close to $R_1$ (say, such that $\delta(r_1)\leq 1/2$), there exist $C_1,C_2>0$ and $n_0 \geq 1$ depending on $r_1,r_2,\Lambda$ but not $n$, such that
\begin{equation}
\label{eq:5:minusbound}
	C_1\leq\|\psi_n^-\|_{L^\infty(B_r)}\leq C_2
\end{equation}
for all $n\geq n_0$ and $r\in (r_1,r_2)$.

This gives Theorem~\ref{th:3:necest} in the special case of the functions $\psi_n^-$. The final step in the proof is therefore to estimate $\|\psi_n^+\|_{L^\infty(B_r)}$ in terms of $\|\psi_n^-\|_{L^\infty(B_r)}$, and vice versa; Theorem~\ref{th:3:necest} will follow immediately from \eqref{eq:5:minusbound} and the next lemma. Here we use an argument of D.~Mangoubi \cite{mangoubi:08}, combining a weak Harnack inequality and a weak maximum principle from \cite{gilbarg:83}.

\begin{lemma}
\label{lemma:5:twowaybounds}
Let $R_0<r_1<r_2<R_1$. There exist $C>0$ and $n_0\geq 1$ depending only on $r_1,r_2$ and $\Lambda$, the upper bound for $\mu_n$, such that
\begin{displaymath}
\begin{aligned}
	\|\psi_n^-\|_{L^\infty(B_{r_1})} &\leq C \|\psi_n^+\|_{L^\infty(B_{r_2})}\\
	\|\psi_n^+\|_{L^\infty(B_{r_1})} &\leq C \|\psi_n^-\|_{L^\infty(B_{r_2})}
\end{aligned}
\end{displaymath}
for all $n \geq n_0$.
\end{lemma}

\begin{proof}
(Cf.~\cite[Theorem~3.4]{mangoubi:08}.) We first remark that, given $r_1>R_0$, for $n$ sufficiently large, $r_1>R_n$, that is, $\psi_n$ must change sign in $B_{r_1}$ (see Lemma~\ref{lemma:3:nodalset}(ii) and note that by choice of $\varepsilon(n)$, $R_n=R_{n,\varepsilon(n)}\to R_0$). We write $L_n \psi_n := (\Delta + \mu_n)\psi_n = 0$, and given $r_1$ and $r_2$, $M_n:= \sup_{B_{r_2}}\psi_n^+$. Consider the function $\varphi_n := M_n - \psi_n$. We have $L_n\varphi_n=(\Delta+\mu_n)(M_n-\psi_n)=\mu_n M_n$, so that, if $\psi_n\leq 0$, then $0\leq -\psi_n \leq M_n-\psi_n=\varphi_n$; if $\psi_n\geq 0$, then $\varphi_n \geq 0$ still, while also
\begin{displaymath}
	0\leq L_n\varphi_n \leq \mu_n M_n \leq \Lambda M_n.
\end{displaymath}
We apply the weak Harnack inequality \cite[Theorem~8.18 or 9.22]{gilbarg:83}, suitably rescaled, to $L_n\varphi_n \leq \Lambda M_n$ on the ball $B:=B_{(r_1+r_2)/2}$. Thus there exist positive constants $p,C$ depending only on $r_1,r_2,\Lambda$ such that
\begin{equation}
\label{eq:5:pbound}
\Bigl(\frac{1}{|B|}\int_B {\varphi_n}^p\,dx\Bigr)^\frac{1}{p} \leq C(\inf_B \varphi_n+\frac{r_1+r_2}{2}\|\Lambda
	M_n\|_{L^2(B)}).
\end{equation}
Absorbing $(r_1+r_2)/2\|\Lambda\|_{L^2(B)}$ into the constant $C$, and using the bound
\begin{displaymath}
	\inf_B \varphi_n = \inf_B (M_n-\psi_n)=M_n+\inf_B(-\psi_n)\leq M_n
\end{displaymath}
since $\psi_n=0$ somewhere in $B$ (as $R_n<r_1$), we rewrite \eqref{eq:5:pbound} as
\begin{equation}
\label{eq:5:pbound2}
\Bigl(\frac{1}{|B|}\int_B {\varphi_n}^p\,dx\Bigr)^\frac{1}{p} \leq \widetilde CM_n,
\end{equation}
where $\widetilde C$ and $p$ depend only on $r_1,r_2,\Lambda$.

We now apply \cite[Theorem~8.17 or 9.20]{gilbarg:83}, again rescaled, to $L_n\varphi_n \geq 0$ on $B_{r_1}$ and $B=B_{(r_1+r_2)/2}$. That is, for any $p\in (0,\infty)$, there exists $K=K(p,r_1,r_2,\Lambda)>0$ such that
\begin{displaymath}
\sup_{B_{r_1}}\varphi_n\leq K\Bigl(\frac{1}{|B|}\int_B {\varphi_n}^p\,dx\Bigr)^\frac{1}{p} \leq K\widetilde CM_n
\end{displaymath}
by \eqref{eq:5:pbound2}, if we choose the $p\in (0,\infty)$ for which \eqref{eq:5:pbound2} holds (independent of $\psi_n, M_n$). Rewriting this,
\begin{displaymath}
\sup_{B_{r_1}}(-\psi_n) \leq \sup_{B_{r_1}}(M_n-\psi_n)\leq K\widetilde CM_n=K\widetilde C\sup_{B_{r_2}}\psi_n^+,
\end{displaymath}
that is,
\begin{displaymath}
\sup_{B_{r_1}} \psi_n^- \leq C(r_1,r_2,\Lambda)\sup_{B_{r_2}}\psi_n^+.
\end{displaymath}
We can obtain the other inequality, with the same constant, by interchanging the r\^oles of $\psi_n^-$ and $\psi_n^+$.
\end{proof}

\bibliographystyle{amsplain}

\providecommand{\bysame}{\leavevmode\hbox to3em{\hrulefill}\thinspace}
%\providecommand{\MR}{\relax\ifhmode\unskip\space\fi MR }
%% \MRhref is called by the amsart/book/proc definition of \MR.
%\providecommand{\MRhref}[2]{%
%  \href{http://www.ams.org/mathscinet-getitem?mr=#1}{#2}
%}
\providecommand{\href}[2]{#2}

\end{document}